\documentclass[a4paper,10pt]{amsart}

\usepackage{amssymb}
\usepackage{amsmath}
\usepackage{amsthm}
\usepackage{color}
\usepackage{enumerate}

\newcommand{\Hom}{\operatorname{Hom}}

\DeclareMathOperator{\Mor}{Mor}

\newcommand{\Ext}{\operatorname{Ext}}

\newcommand{\Ker}{\operatorname{Ker}}

\newcommand{\Coker}{\operatorname{Coker}}

\newcommand{\cone}{\operatorname{cone}}

\newcommand{\C}{\mathbf{C}}

\newcommand{\Modr}[1]{\mathrm{Mod}\textrm{-}{#1}}

\newcommand{\Proj}{\mathrm{Proj}}

\DeclareMathOperator{\GProj}{GProj}

\DeclareMathOperator{\Z}{Z}
\renewcommand{\H}{\textrm{H}}

\theoremstyle{plain}
\newtheorem{theorem}{Theorem}[section]
\newtheorem{lemma}[theorem]{Lemma}

\newtheorem{corollary}[theorem]{Corollary}

\theoremstyle{definition}

\theoremstyle{remark}
\newtheorem{remark}[theorem]{Remark}

\usepackage{tikz}
\usepackage{tikz-cd}

\title[Verdier quotients and Gorenstein-projective precovers]{Verdier quotients of homotopy categories of rings and Gorenstein-projective precovers}

\author{Manuel Cort\'es-Izurdiaga}
\address{Departamento de Matemática Aplicada, Universidad de Málaga, 29071, Málaga, Spain}
\thanks{Partially supported by the Spanish Government under grants PID2020-113206GB-I00, funded by MCIN/AEI/10.13039/501100011033, and by Junta de Andalucia under grant P20-00770}
\email{mizurdiaga@uma.es}



\begin{document}

\begin{abstract}
Let $R$ be a ring, $\Proj$ be the class of all projective right $R$-modules, $\mathcal K$ be the full subcategory of the homotopy category $\mathbf K(\Proj)$ whose class of objects consists of all totally acyclic complexes, and $\Mor_{\mathcal K}$ be the class of all morphisms in $\mathbf K(\Proj)$ whose cones belong to $\mathcal K$. We prove that if $\mathbf K(\Proj)$ has enough $\Mor_{\mathcal K}$-injective objects, then the Verdier quotient $\mathbf K(\Proj)/\mathcal K$ has small Hom-sets, and this last condition implies the existence of Gorenstein-projective precovers in $\Modr R$ and of totally acyclic precovers in $\mathbf C(\Modr R)$.
\end{abstract}

\maketitle

\section{Introduction}

The aim of relative homological algebra is to extend the homological techniques to those abelian categories $\mathcal A$ that do not have enough projective or injective objects. One way to do this is to substitute the class of projective or injective objects by a class of objects $\mathcal C$, and then compute left and right resolutions using the objects in $\mathcal C$, see \cite{EilenbergMoore} and \cite{EnochsJenda}. In order to get unique (up to homotopy) resolutions, we need the existence of $\mathcal C$-precovers or right $\mathcal C$-approximations, and of $\mathcal C$-preenvelopes or left $\mathcal C$-approximations (recall that a $\mathcal C$-precover of an object $A$ of $\mathcal A$ is a morphism $f:C \rightarrow A$ with $C \in \mathcal C$ and such that $\Hom_{\mathcal A}(C',f)$ is an epimorphism for each $C' \in \mathcal C$, and that $\mathcal C$-preenvelopes are defined dually). That is, it is possible to compute derived functors in the abelian category $\mathcal A$ without enough projective objects (resp. injectives objects) if one has a precovering (resp. preenveloping) class $\mathcal C$. From this point of view, the crucial point is to prove the existence of precovers and preenvelopes.

One of the most important ``relative homological algebras" is the so-called Gorenstein homological algebra. In this one, the projective, injective and flat modules are substituted by the Gorenstein-projective, Gorenstein-injective and Gorenstein-flat modules. The Gorenstein homological algebra is important, for instance, in commutative algebra, where there exist non-regular noetherian local rings (i. e., rings with infinite global dimension), but with finite Gorenstein global dimension, the Gorenstein rings (see \cite{Lars} for a treatment about Gorenstein dimensions over commutative rings). On the other hand, the Gorenstein categories introduced in \cite{EnochsEstradaGarciaRozas} might not have enough projectives, but they do have enough Gorenstein-projective objects, so that we can use Gorenstein-projective resolutions to compute derived functors.

If we look at the category $\Modr R$ of right modules over a not necessarily commutative ring $R$, there have been a lot of research since the earlies 90s trying to extend the Gorenstein homological algebra to $\Modr R$. As we mentioned before, the first (and key) step is to prove the existence of Gorenstein-projective and Gorenstein-flat precovers, and Gorenstein-injective preenvelopes. In \cite{YangLiang}, Yang and Liang showed that every module has a Gorenstein-flat precover, and in \cite{SarochStovicek}, \v{S}aroch and Stovicek proved the existence of Gorenstein-injective preenvelopes. However, it is not known if every module has a Gorenstein-projective precover.

The existence of Gorenstein-projective precovers is known for certain classes of rings. The first class is the so-called right Gorenstein regular rings \cite{EnochsCortesTorrecillas}, which are those rings for which the category of right modules is Gorenstein in the sense of \cite{EnochsEstradaGarciaRozas}. In this case, the existence of Gorenstein-projective precovers follows from the more general result \cite[Theorem 2.26]{EnochsEstradaGarciaRozas}. The second class consists of the left coherent and right $n$-perfect rings (those rings for which the flat right modules have finite projective dimension). The existence of Gorenstein-projective precovers over these rings was proved in \cite{EstradaIacobYeomans}, where Estrada, Iacob and Yeomans, using module-theoretic methods, extended previous results obtained with homotopy category methods by J\o rgensen in \cite{Jorgensen} and Murfet and Salarian in \cite{MurfetSalarian}. Finally, \v{S}aroch and the author have recently proven \cite[Corollaries 5.10 and 6.5]{CortesSaroch} that there exist Gorenstein-projective precovers over any right $\Sigma$-pure-injective ring (more generally, over a right $\lambda$-$\Sigma$-pure-injective ring, where $\lambda$ is an infinite regular cardinal), where a ring is $\Sigma$-pure-injective if every pure monomorphism $P \rightarrow Q$ with $P$ projective, splits.

The key fact in J{\o}rgensen's result about Gorenstein-projective precovers is that he assumes that the full subcategory $\mathcal K$ of $\mathbf K(\Proj)$ consisting of all totally acyclic complexes is coreflective. In this paper we continue with the investigation of the relationship between $\mathcal K$ and the existence of Gorenstein-projective precovers. We prove that if the Verdier quotient of $\mathbf K(\Proj)$ by $\mathcal K$ has small Hom-sets, then $\mathcal K$ is a coreflective subcategory of $\mathbf K(\Proj)$ (see Theorem \ref{t:coreflective}) and, by J{\o}rgensen's result, Gorenstein-projective precovers exist in $\Modr R$. Moreover, we prove that totally acyclic precovers exist in $\mathbf C(\Modr R)$ as well (Corollary \ref{c:TotallyAcyclicPrecovers}). 

These results lead us to look for conditions that imply that the Verdier quotient of $\mathbf K(\Proj)$ by $\mathcal K$ has small Hom-sets. In Section 4 we prove that if the class $\Mor_{\mathcal K}$ of those morphisms of $\mathbf K(\Proj)$ whose cones belong to $\mathcal K$ satisfies a certain generalization of the Baer's Lemma for injectivity, then the aforementioned Verdier quotient has small Hom-sets (see Theorem \ref{t:ExistenceInjective}).

\section{Preliminaries}


An cardinal $\kappa$ is \textit{singular} if there exists a family of cardinals smaller than $\kappa$, $\{\lambda_\alpha\mid \alpha < \mu\}$, where $\mu$ is a cardinal smaller than $\kappa$, such that $\kappa = \sum_{\alpha<\kappa}\lambda_\alpha$ (cardinal sum). The cardinal $\kappa$ is called \textit{regular} if it is not singular.

Let $\mathcal A$ be an additive category and $\mathcal M$ a class of morphisms in $\mathcal A$. An object $E$ of $\mathcal A$ is called $\mathcal M$-injective if $\Hom_{\mathcal A}(i,E)$ is an epimorphism in the category of abelian groups for each $i \in \mathcal M$. We say that $\mathcal A$ has \textit{enough $\mathcal M$-injective objects} if for every object $A$ of $\mathcal A$, there exists a morphism $i:A \rightarrow E$ belonging to $\mathcal M$ and with $E$ being an $\mathcal M$-injective object. Dually they are defined the notions of \textit{$\mathcal M$-projective} objects and of the existence of \textit{enough $\mathcal M$-projective objects}.

We say that the class $\mathcal M$ satisfies the \textit{generalized Baer Lemma} if there exists a subset $\mathcal N$ of $\mathcal M$ such that any $\mathcal N$-injective object of $\mathcal A$ is $\mathcal M$-injective.

Given a full subcategory $\mathcal B$ of $\mathcal A$, a \textit{$\mathcal B$-precover} of an object $A \in \mathcal A$ is a morphism $f:B \rightarrow A$ with $B \in \mathcal B$ such that every $B'\in \mathcal B$ is $f$-projective. The $\mathcal B$-precover $f$ is a \textit{$\mathcal B$-coreflection} if $\Hom_{\mathcal A}(B',f)$ is an isomorphism for each $B' \in \mathcal B$. If every object of $\mathcal A$ has a $\mathcal B$-precover (resp. a $\mathcal B$-coreflection), then $\mathcal B$ is called \textit{precovering} (resp. \textit{coreflective}). Being coreflective is equivalent to the inclusion functor  $\mathcal B \rightarrow \mathcal A$ having a right adjoint \cite[Theorem IV.2]{Maclane}. Dually, they are defined the notions related with $\mathcal B$-preenvelopes and $\mathcal B$-reflections.

Suppose that $\mathcal A$ has direct limits. A \textit{transfinite sequence} in $\mathcal A$ is just a direct system, $(X_\alpha,u_{\alpha\beta} \mid \alpha < \beta < \kappa)$, indexed by an ordinal $\kappa$, such that the canonical morphism $\displaystyle \varinjlim_{\alpha < \beta}X_\alpha \rightarrow X_\beta$ is an isomorphism for each limit ordinal $\beta < \kappa$. The \textit{transfinite composition} of the sequence is the canonical morphism $u_0:X_0 \rightarrow \varinjlim X_\alpha$. Furthermore, if $\mathcal A$ has cokernels, the transfinite sequence $(X_\alpha,u_{\alpha\beta}\mid \alpha < \beta < \kappa)$ is called a \textit{$\mathcal B$-filtration} if $u_{\alpha\beta}$ is a monomorphism and $\Coker u_{\alpha\alpha+1} \in \mathcal B$ for each $\alpha < \beta < \kappa$.

Throughout this paper, $R$ denotes a not necessarily commutative ring with unit and all modules are right $R$-modules. The category of all right $R$-modules is denoted by $\Modr R$, and the full subcategory of all projective modules, by $\Proj$. Given any class of modules $\mathcal C$, we consider the left and right orthogonal classes of $\mathcal C$ with respect to $\Ext$, i. e., 
\begin{displaymath}
\mathcal C^\perp = \{M \in \Modr R \mid \Ext^1_R(C,M)=0 \, \forall C \in \mathcal C\}
\end{displaymath}
and
\begin{displaymath}
{^\perp}\mathcal C = \{M \in \Modr R \mid \Ext^1_R(M,C)=0 \, \forall C \in \mathcal C\}.
\end{displaymath}

We denote by $\mathbf C(\Modr R)$ the category whose class of objects consists of all cochain complexes $X$ of modules,
\begin{displaymath}
\begin{tikzcd}
\arrow{r} \cdots & X^{n-1} \arrow{r}{d_X^{n-1}} & X^n \arrow{r}{d_X^n} & X^{n+1} \arrow{r} & \cdots
\end{tikzcd}
\end{displaymath}
and whose morphisms are the cochain maps. Given such a complex $X$ and $n \in \mathbb Z$, we denote by $\Z^n(X)$ the $nth$-cycle of $X$ for each $n \in \mathbb Z$, i. e., $\Z^n(X)=\Ker d^n_X$, and by $\H^n(X)$ the \textit{$nth$-homology module}. The \textit{$nth$-shift} of $X$ is the complex $X[n]$ such that $X[n]^k=X^{k-n}$ and $d^k_{X[n]}=(-1)^nd_X^{k-n}$ for each $k \in \mathbb Z$. The \textit{cone} of a cochain map $f:X \rightarrow Y$ is the complex $\cone(f)$ such that $\cone(f)^n=X^{n+1} \oplus Y^n$ and $d_{\cone(f)}^n(x,y)=(-d_X^{n+1}(x),d^n_Y(y)-f^{n+1}(x))$ for each $x \in X^{n+1}$ and $y \in Y^n$. A short exact sequences of complexes,
\begin{displaymath}
\begin{tikzcd}
0 \arrow{r} & X \arrow{r}{f} & Y \arrow{r}{g} & Z \arrow{r} & 0
\end{tikzcd}
\end{displaymath}
is called \textit{semi-split} if $f^n$ is a split monomorphism (equivalently, $g^n$ is a split epimorphism) for each $n \in \mathbb Z$. In this case, we say that $f$ is a \textit{semi-split monomorphism} and that $g$ is a \textit{semi-split epimorphism}. Recall that $\mathbf C(\Modr R)$ with the class of all semi-split short exact sequences is an exact category in the sense of Quillen (see, for instance, \cite[Example 13.8]{Buhler}). Any chain homotopy $s$ between two cochain maps $f$ and $g$ is denoted $s: f \Rightarrow g$.

Given a module $M$ and an integer $n$, the \textit{disk on $M$ of degree $n$} is the complex $D_n(M)$ such that $D_n(M)^n=D_n(M)^{n+1}=M$, $D_n(M)^k=0$ if $k \notin\{n,n+1\}$, $d_{D_n(M)}^n=1_M$, and $d_{D_n(M)}^k=0$ if $k\neq n$. A complex $X$ is \textit{contractible} if it is a direct sum of disks (on certain modules).

The complex $X$ is called \textit{acyclic} if it is exact, and \textit{totally acyclic} if it is acyclic, $\Hom_R(-,\Proj)$-exact (that is, the complex $\Hom_R(X,P)$ is acyclic for every $P \in \Proj$) and $X^n \in \Proj$ for every $n \in \mathbb Z$. It is well known that an acyclic complex $X$ of projective modules is totally acyclic if and only if $\Z^n(X) \in {^\perp}\Proj$ for each $n \in \mathbb Z$. Finally, a module $M$ is \textit{Gorenstein-projective} if there exists a totally acyclic complex $X$ such that $\Z^0(X) \cong M$. We denote by $\GProj$ the full subcategory of $\Modr R$ of all Gorenstein-projective modules.

We denote by $\mathbf K(\Modr R)$ the homotopy category of $R$, that is, the category whose class of objects is $\mathbf C(\Modr R)$ and whose morphisms are homotopy equivalence classes of cochain maps. Given a cochain map $g$, we use the notation $\underline g$ for the corresponding morphism in $\mathbf K(\Modr R)$. The \textit{strict triangles} in $\mathbf K(\Modr R)$ are of the form
\begin{displaymath}
\begin{tikzcd}
X \arrow{r}{\underline f} & Y \arrow{r}{\underline u} & \cone(f) \arrow{r}{\underline v} & X[-1],
\end{tikzcd}
\end{displaymath}
where $f:X \rightarrow Y$ is a fixed cochain map, $u$ is the canonical injection $Y \rightarrow \cone(f)$, and $v$ is the projection $\cone(f) \rightarrow X[-1]$. Then, the category $\mathbf K(\Modr R)$ is a triangulated category with the class of \textit{exact triangles} consisting of those triangles isomorphic to a strict triangle.

By a \textit{cone} of a morphism $f:X \rightarrow Y$ in $\mathbf K(\Modr R)$ we mean an object $Z$ such that there exists an exact triangle of the form
\begin{equation}\label{e:Triangle}
\begin{tikzcd}
X \arrow{r}{f} & Y \arrow{r} & Z \arrow{r} & X[-1].\tag{!}
\end{tikzcd}
\end{equation}
A full subcategory $\mathcal D$ of $\mathbf K(\Modr R)$ is \textit{closed under cones} if every exact triangle as (\ref{e:Triangle}) with $X,Y \in \mathcal D$ satisfies that $Z\in \mathcal D$ as well. The subcategory $\mathcal D$ is called \textit{triangulated} if it is closed under isomorphisms, cones and $\mathcal D[-1]=\mathcal D$; \textit{thick} if it is triangulated and closed under direct summands, and \textit{localising}, if it is thick and closed under arbitrary direct sums. Notice that triangulated subcategories of $\mathbf K(\Modr R)$ are localization of full additive subcategories $\mathcal Y$ of $\mathbf C(\Modr R)$ (that is,  subcategories containing the zero object and closed under finite direct sums), which are closed under translations and the formation of mapping cones (see \cite[Corollary 10.2.5]{Weibel}). We denote by $\mathbf K(\Proj)$ the full subcategory of $\mathbf K(\Modr R)$ whose class of objects are all complexes of projective modules, and by $\mathcal K$ the full subcategory of $\mathbf K(\Proj)$ whose class of objects are the totally acyclic complexes.

Let $\mathcal T$ be a triangulated subcategory of $\mathbf K(\Modr R)$ and $\mathcal D$ a thick subcategory of $\mathcal T$. Following \cite[Definition 1.5.3]{Neeman10}, we denote by $\Mor_\mathcal D$ the class of morphisms $f:X \rightarrow Y$ in $\mathcal T$ for which there exists a triangle
\begin{displaymath}
\begin{tikzcd}
X \arrow{r}{f} & Y \arrow{r} & Z \arrow{r} & X[-1]
\end{tikzcd}
\end{displaymath}
with $Z \in \mathcal D$.
The Verdier quotient of $\mathcal T$ by $\mathcal D$ is the category $\mathcal T/\mathcal D$ whose class of objects coincides with the class of objects of $\mathcal T$, and whose morphisms between two objects $X$ and $Y$ of $\mathcal T$ are equivalence classes of triples $(Z,f,g)$ such that $f:Z \rightarrow X$ is a morphism in $\Mor_{\mathcal D}$ and $g:Z \rightarrow Y$ is a morphism in $\mathcal T$. The equivalence relation is defined in the following way: $(Z,f,g) \sim (Z',f',g')$ if and only if there exists a triple $(Z'',f'',g'')$ and morphisms $u:Z'' \rightarrow Z$ and $v:Z''\rightarrow Z'$ which make the diagram
\begin{displaymath}
\begin{tikzcd}
 & Z' \arrow{dl}{f'} \arrow{dr}{g'} & \\
X & Z'' \arrow{l}{f''} \arrow{u}{v} \arrow{r}{g''} \arrow{d}{u} & Y\\
 & Z \arrow{ul}{f} \arrow{ur}{g} & \\
\end{tikzcd}
\end{displaymath}
commutative. In general, the class of morphisms between two objects of the Verdier quotient of $\mathcal T$ by $\mathcal D$ is not a small set, just a proper class.

\section{Verdier quotients and precovers}

In this section, we establish the relationship between the Verdier quotient\\ $\mathbf K(\Proj)/\mathcal K$ and the existence of Gorenstein-projective precovers in $\Modr R$ and of totally acyclic precovers in $\mathbf C(\Modr R)$. First, we give the following lemma, which will be used several times in the sequel. It is essentially proved in \cite[Construction 2.4]{Jorgensen} for the class of all totally acyclic complexes. We will use the following more general form.

\begin{lemma}\label{l:Semisplit_epimorphism}
Let $\mathcal X \subseteq \mathcal Y$ be full subcategories of $\C(\Modr R)$. Suppose that $\mathcal X$ is closed under countable direct sums and that, for every $Y \in \mathcal Y$ and $n \in \mathbb Z$, the disk on $Y^n$ of degree $n$, $D_n(Y^n)$, belongs to $\mathcal X$. Let $f:X \rightarrow Y$ be a morphism in $\mathcal Y$. Then:

\begin{enumerate}
\item If $X \in \mathcal X$, there exists a contractible complex $Z \in \mathcal X$ and a semi-split epimorphism $g:X \oplus Z \rightarrow Y$ in $\mathcal Y$ such that and $\underline {gh} = \underline f$, where $\underline h:X \rightarrow X \oplus Z$ is the isomorphism in $\mathbf K(\Modr R)$ induced by the inclusion $X \rightarrow X \oplus Z$ in $\mathbf C(\Modr R)$.

\item If $Y \in \mathcal X$, there exists a contractible complex $Z \in \mathcal X$ and a semi-split monomorphism $g:X \rightarrow Y \oplus Z$  in $\mathcal Y$, such that and $\underline {hg} = \underline f$, where $\underline h:Y \oplus Z \rightarrow Y$ is the isomorphism in $\mathbf K(\Modr R)$ induced by the projection $Y \oplus Z \rightarrow Y$ in $\mathbf C(\Modr R)$.
\end{enumerate}
\end{lemma}

\begin{proof}
We proof (1), since (2) is dual. For each $n \in \mathbb Z$, there exists a cochain map $f_n:D_n(Y^n) \rightarrow Y$ given by $f_n^k=0$ if $k \notin \{n,n+1\}$, $f_n^n=1_{Y_n}$ and $f_n^{n+1}=d_{Y}^n$. Then the morphism $g:X \oplus(\oplus_{n \in \mathbb Z}D_n(Y^n)) \rightarrow Y$ given by the direct sum of the $f_n's$ and $f$ is the desired morphism.
\end{proof}

The key property of $\mathcal K$ regarding the Verdier quotient is that it is a localising subcategory of $\mathbf K(\Modr R)$.

\begin{lemma}\label{p:Localising}
The subcategory $\mathcal K$ of $\mathbf K(\Modr R)$, whose class of objects consists of all totally acyclic complexes, is localising.
\end{lemma}

\begin{proof}
Notice that trivially $\mathcal K[-1] = \mathcal K$. Moreover, $\mathcal K$ is closed under direct sums, since direct sums in $\mathbf K(\Proj)$ are computed in $\C(\Modr R)$ and the class of totally acyclic complexes is closed under direct sums in $\C(\Modr R)$.

Let us show that $\mathcal K$ is closed under direct summands. Let $E \in \mathcal K$ and $X \in \mathbf K(\Proj)$ be a direct summand of $E$. Then there exist morphisms in $\mathbf K(\Proj)$, $\underline f:X \rightarrow E$ and $\underline g:E \rightarrow X$, such that $\underline g\underline f=1_X$ in $\mathbf K(\Proj)$. In particular, if $f_*$ and $g_*$ are the induced morphisms between the homology modules, $g_*^nf_*^n=1_{\H^n(X)}$, which means that $\H^n(X)$ is isomorphic to a direct summand of $\H^n(E)$ for each $n \in \mathbb Z$. This implies that $X$ is acyclic.

In order to prove that $X$ is totally acyclic, let $Q$ be a projective module and let us see that $\H^n (\Hom_R(X,Q))=0$ for each $n \in \mathbb Z$. Notice that $0=\H^n (\Hom_R(E,Q)) \cong \Hom_{\mathbf K(\Modr R)}(E,S_n(Q))$, where $S_n(Q)$ is the complex with $Q$ in degree $n$ and zeros elsewhere. Since $X$ is a direct summand of $E$ in $\mathbf K(\Modr R)$, $\Hom_{\mathbf K(\Modr R)}(X,S_n(Q))=0$ as well, which implies, using the same isomorphism as above, that $\H^n(\Hom_R(X,Q))=0$ for each $n \in \mathbb Z$. Then $X$ is totally acyclic.

In order to show that $\mathcal K$ is closed under cones, let us prove that the class of totally acyclic complexes is closed under cones in $\mathbf C(\Modr R)$. Let $f:X \rightarrow Y$ be a morphism between totally acyclic complexes. Then $\cone(f)^n \in \Proj$ for each $n \in \mathbb Z$, and $\cone(f)$ is acyclic by \cite[Corollary 1.5.4]{Weibel}, since $f$ is a quasi-isomorphism. In order to show that $\Z^n(\cone(f)) \in {^\perp}\Proj$, notice that 
\begin{displaymath}
\Z^n(\cone(f)) = \{(x,y)\in \Z^{n+1}(X) \oplus Y^n \mid d_Y^n(y)=f^{n+1}(x)\},
\end{displaymath}
so that there exists a commutative diagram with exact rows,
\begin{displaymath}
\begin{tikzcd}
0 \arrow{r} & \Z^n(Y) \arrow{r} \arrow[equal]{d} & \Z^n(\cone(f)) \arrow{r} \arrow{d} & \Z^{n+1}(X) \arrow{d}{f^{n+1}} \arrow{r} & 0\\
0 \arrow{r} & \Z^n(Y) \arrow{r} & Y^n \arrow{r}{d_Y^n} & \Z^{n+1}(Y) \arrow{r} & 0
\end{tikzcd}
\end{displaymath}
which follows from the fact that the right hand square is a pullback (see \cite[10.3]{Wisbauer}). Then $\Z^n(\cone(f)) \in {^\perp}\Proj$, since this class is closed under extensions.

Finally, using this and the fact that $\mathcal K$ is closed under isomorphisms, it follows that $\mathcal K$ is closed under cones, which concludes the proof.
\end{proof}

In \cite[Theorem 1.10]{Jorgensen}, J\o rgensen proved that $\mathcal K$ is coreflective in $\mathbf K(\Proj)$ provided that (1) $R$ is a noetherian commutative ring with a dualizing complex or (2) $R$ is a left coherent and right noetherian $K$-algebra over the field $K$ for which there exists a left noetherian $K$-algebra $B$ and a dualizing complex ${_B}D_A$. In p. 66 J\o rgensen asked whether this is true for any ring $R$.                                                                                                                                                                     Our next result gives an affirmative answer to this question when the Verdier quotient $\mathbf K(\Proj)/\mathcal K$ has small Hom-sets.

\begin{theorem}\label{t:coreflective}
Suppose that the Verdier quotient $\mathbf K(\Proj)/\mathcal K$ has small Hom-sets. Then, $\mathcal K$ is a coreflective subcategory of $\mathbf K(\Modr R)$.
\end{theorem}

\begin{proof}
First, we see that $\mathcal K$ is a coreflective subcategory of $\mathbf K(\Proj)$. By \cite[Theorem 1.1]{Neeman08}, $\mathbf K(\Proj)$ satisfies Brown representability. By \cite[Example 8.4.5]{Neeman01} and Proposition \ref{p:Localising}, a Bousfield localisation functor exists for the pair $\mathcal K \subseteq \mathbf K(\Proj)$. But, by \cite[Proposition 9.1.18]{Neeman01}, this is equivalent to $\mathcal K$ being coreflective in $\mathbf K(\Proj)$.

On the other hand, as argued in \cite[Theorem 3.2]{Neeman10}, $\mathbf K(\Proj)$ is a coreflective subcategory of $\mathbf K(\Modr R)$. Then, $\mathcal K$ is a coreflective subcategory of $\mathbf K(\Modr R)$ as well.
\end{proof}

Using \cite[Corollary 2.3]{Jorgensen} (see \cite[Setup 2.1]{Jorgensen}), we get the existence of Gorenstein-projective precovers when the Verdier quotient $\mathbf K(\Proj)/\mathcal K$ has small Hom-sets:

\begin{corollary}\label{c:GorensteinPrecovers}
If the Verdier quotient $\mathbf K(\Proj)/\mathcal K$ has small Hom-sets, then every module has a Gorenstein-projective precover.
\end{corollary}


Now we establish the relationship between the Verdier quotient $\mathbf K(\Proj)/\mathcal K$ and the existence of precovers by totally acyclic complexes in $\mathbf C(\Proj)$. Let us discuss first the relationship between precovering and coreflective subcategories. It is a classical result that if $\mathcal B$ is a coreflective subcategory of a category $\mathcal A$, then $\mathcal B$ is precovering. In particular, if $\mathcal K$ is coreflective in $\mathbf K(\Proj)$, then $\mathcal K$ is precovering in $\mathbf K(\Modr R)$. But, is the class of all totally acyclic complexes precovering in $\mathbf C(\Modr R)$? The answer to this question is essentially given by the following result, which is inspired in \cite[Construction 2.4 and Lemma 2.6]{Jorgensen}:

\begin{theorem}\label{t:Precovers}
Let $\mathcal X \subseteq \mathcal Y$ be full subcategories of $\C(\Modr R)$. Denote by $\mathcal L$ and $\mathcal D$ the full subcategories of $\mathbf K(\Modr R)$ whose class of objects are $\mathcal X$ and $\mathcal Y$, respectively. Suppose that $\mathcal X$ is closed under countable direct sums, $\mathcal X[-1] \subseteq \mathcal X$ and that, for each $Y \in \mathcal Y$ and $n \in \mathbb Z$, $D_n(Y^n)$ belongs to $\mathcal X$. The following assertions are equivalent:
\begin{enumerate}
\item $\mathcal L$ is a coreflective subcategory of $\mathcal D$.

\item For each complex $C \in \mathcal Y$ there exists a semi-split exact sequence in $\mathbf C(\Modr R)$,
\begin{displaymath}
\begin{tikzcd}
0 \arrow{r} & Z \arrow{r} & X \arrow{r} & C \arrow{r} & 0,
\end{tikzcd}
\end{displaymath}
such that $X \in \mathcal X$ and $\Hom_{\mathbf{K}(\Modr R)}(\widehat X,Z)=0$ for each $\widehat X \in \mathcal X$.
\end{enumerate}

Moreover, when both assertions hold, the class $\mathcal X$ is precovering in $\mathcal Y$.
\end{theorem}

\begin{proof}
(1) $\Rightarrow$ (2). We use the argument of \cite[Lemma 2.6]{Jorgensen}. Let $q$ be the right adjoint of the inclusion functor $i:\mathcal L \rightarrow \mathcal D$, and let $\nu$ be the counit of the adjunction. Let $C$ be a complex in $\mathcal Y$ and consider the morphism $\nu_C:q(C) \rightarrow C$. By Lemma \ref{l:Semisplit_epimorphism}, we can find a complex $X \in \mathcal X$, an isomorphism $\underline h: X \rightarrow q(C)$ in $\mathcal D$, and a semi-split epimorphism $f:X \rightarrow C$ in $\C(\Modr R)$ such that $\nu_C\underline h=\underline f$. Then, it is easy to check that $\underline f$ is an $\mathcal L$-coreflection of $C$ in $\mathcal D$.

Now, consider the short exact sequence in $\C(\Modr R)$,
\begin{displaymath}
\begin{tikzcd}
0 \arrow{r} & Z \arrow{r}{g} & X \arrow{r}{f} & C \arrow{r} & 0
\end{tikzcd}
\end{displaymath}
and notice that, since it is semi-split, it induces the exact triangle in $\mathbf K(\Modr R)$,
\begin{displaymath}
\begin{tikzcd}
Z \arrow{r}{\underline g} & X \arrow{r}{\underline f} & C \arrow{r} & Z[-1]
\end{tikzcd}
\end{displaymath}
Given $\widehat X \in \mathcal X$, since $\Hom_{\mathbf K(\Modr R)}(\widehat X,-)$ is a homological functor by \cite[Lemma 1.1.10]{Neeman01}, there exists an exact sequence of abelian groups
\begin{displaymath}
\begin{tikzcd}
\Hom_{\mathbf K(\Modr R)}(\widehat X,Z) \arrow{r}{\beta_2} & \Hom_{\mathbf K(\Modr R)}(\widehat X,X) \arrow{r}{\alpha_2} & \Hom_{\mathbf K(\Modr R)}(\widehat X,C)
\end{tikzcd}
\end{displaymath}
in which $\alpha_2=\Hom_{\mathbf K(\Modr R)}(\widehat X,\underline f)$ is an isomorphism; then $\beta_2=0$. Analogously, there exists an exact sequence
\begin{displaymath}
\begin{tikzcd}
\Hom_{\mathbf K(\Modr R)}(\widehat X,X[1]) \arrow{r}{\alpha_1} & \Hom_{\mathbf K(\Modr R)}(\widehat X,C[1]) \arrow{r}{\beta_1} & \Hom_{\mathbf K(\Modr R)}(\widehat X,Z).
\end{tikzcd}
\end{displaymath}
Using that $\mathcal X[-1] \subseteq \mathcal X$, it is easy to check that $\alpha_1$ is an isomorphisms as well; then, $\beta_1=0$. Now, the exactness of
\begin{displaymath}
\begin{tikzcd}
\Hom_{\mathbf K(\Modr R)}(\widehat X,C[1]) \arrow{r}{\beta_1} & \Hom_{\mathbf K(\Modr R)}(\widehat X,Z) \arrow{r}{\beta_2} & \Hom_{\mathbf K(\Modr R)}(\widehat X,X)
\end{tikzcd}
\end{displaymath}
gives that $\Hom_{\mathbf K(\Modr R)}(\widehat X,Z)=0$, as desired.

(2) $\Rightarrow$ (1). Given $C \in \mathcal Y$, consider the sequence given by (2) and notice that, since it is semi-split, it induces the exact triangle in $\mathbf K(\Modr R)$, 
\begin{displaymath}
\begin{tikzcd}
Z \arrow{r}{\underline g} & X \arrow{r}{\underline f} & C \arrow{r} & Z[-1]
\end{tikzcd}
\end{displaymath}
Given $\widehat X \in \mathcal X$ and applying the homological functor $\Hom_{\mathbf K(\Modr R)}(\widehat X,-)$ to this triangle, we obtain the following exact sequence of abelian groups
\begin{displaymath}
\begin{tikzcd}
& 0=\Hom_{\mathbf K(\Modr R)}(\widehat X,Z) \arrow{r} & \Hom_{\mathbf K(\Modr R)}(\widehat X,X) \arrow{r}{\Hom(\widehat X,\underline f)} & \cdots\\
\cdots \arrow{r}{\Hom(\widehat X,f)} & \Hom_{\mathbf K(\Modr R)}(\widehat X,C) \arrow{r} & \Hom_{\mathbf K(\Modr R)}(\widehat X,Z[-1])=0
\end{tikzcd}
\end{displaymath}
from which it follows that $\Hom_{\mathbf K(\Modr R)}(\widehat X,\underline f)$ is a bijection. Thus, $\underline f$ is an $\mathcal L$-coreflection of $C$.

Now, we prove the last assertion of the theorem. Let $Y \in \mathcal Y$ and take the short exact sequence given by (2):
\begin{equation} \label{e:SES}
\begin{tikzcd}
0 \arrow{r} & Z \arrow{r}{g} & X \arrow{r}{f} & Y \arrow{r} & 0.\tag{\#}
\end{tikzcd}
\end{equation}
We prove that $f$ is an $\mathcal X$-precover of $Y$. Take $h:\widehat X \rightarrow Y$ a morphism in $\mathcal Y$ with $\widehat X \in \mathcal X$. Since $f^n$ is a split epimorphism for each $n \in \mathbb Z$, we can find $k^n:\widehat X^n \rightarrow X^n$ satisfying $f^nk^n=h^n$. Then, it is easy to check that $f^{n+1}(d_X^nk^n-k^{n+1}d_{\widehat X}^n)=0$, and, since (\ref{e:SES}) is exact, there exists $l^n:\widehat X^n \rightarrow Z^{n+1}$ with 
\begin{equation}\label{e:Morphisml}
g^{n+1}l^n=d_X^nk^n-k^{n+1}d_{\widehat X}^n.\tag{\#\#}
\end{equation}
Now, the family $\{l_n\mid n \in \mathbb Z\}$ satisfies that $g^{n+2}d_Z^{n+1}l^n=g^{n+2}l^{n+1}d_{\widehat X}^n$ for each $n \in \mathbb Z$ and, since $g^{n+2}$ is monic, it induces a morphism in $\C(\Modr R)$, $l:\widehat X \rightarrow Z[-1]$. 

Now, use the hypothesis that $\Hom_{\mathbf K(\Modr R)}(\widehat X,Z)=0$ to get a chain contraction, $s:\widehat X \Rightarrow Z$, of $l$. The identity $d_Z^ns^n+s^{n+1}d_{\widehat X}^n=l^n$ gives that $g^{n+1}l^n=d_X^ng^ns^n+g^{n+1}s^{n+1}d_{\widehat X}^n$ which, combined with (\ref{e:Morphisml}), gives
\begin{displaymath}
d_X^n(g^ns^n-k^n)=(g^{n+1}s^{n+1}-k^{n+1})d_{\widehat X}^n
\end{displaymath}
for each $n \in \mathbb Z$. This means that the family of morphisms $\{g^ns^n-k^n\mid n \in \mathbb Z\}$ actually defines a morphism $t:\widehat X \rightarrow X$ in $\mathcal Y$ which trivially satisfies that $ft=h$. This proves that $f$ is an $\mathcal X$-precover of $Y$.
\end{proof}

%

As a consequence of this result we get:

\begin{corollary}\label{c:TotallyAcyclicPrecovers}
Suppose that the Verdier quotient $\mathbf K(\Proj)/\mathcal K$ has small Hom-sets. Then the class of all totally acyclic complexes is precovering in $\mathbf C(\Modr R)$.
\end{corollary}

\begin{proof}
Consider the following full subcategory of $\mathbf C(\Modr R)$,
\begin{displaymath}
\mathcal X = \{P \oplus C \mid P \in \mathbf C(\Proj), C \in \mathbf C(\Modr R) \textrm{ is contractible}\},
\end{displaymath}
and denote by $\mathcal L$ the full subcategory of $\mathbf K(\Modr R)$ whose class of objects is $\mathcal X$. Since every object of $\mathcal L$ is isomorphic to one object of $\mathbf K(\Proj)$, and $\mathbf K(\Proj)$ is coreflective in $\mathbf K(\Modr R)$ (as it is argued in the proof of \cite[Theorem 3.2]{Neeman10}), we conclude that $\mathcal L$ is coreflective in $\mathbf K(\Modr R)$. By Theorem \ref{t:Precovers}, $\mathcal X$ is precovering in $\mathbf C(\Modr R)$. 

Now, by Theorem \ref{t:coreflective}, $\mathcal K$ is coreflective in $\mathbf K(\Proj)$ which implies that $\mathcal K$ is coreflective in $\mathcal L$. Again by Theorem \ref{t:Precovers}, we obtain that the class of totally acyclic complexes is precovering in $\mathcal X$. 

From these two facts it immediately follows that the class of totally acyclic complexes is precovering in $\mathbf C(\Modr R)$.
\end{proof}


\section{Small Hom-sets}

The discussion developed in the preceding section leads us to the problem of when $\mathbf K(\Proj)/\mathcal K$ has small Hom-sets. As mentioned before, not always the Verdier quotient of a triangulated category by a thick subcategory has small Hom-sets, and it is interesting to have criteria that imply this property (see \cite[p. 100]{Neeman01}). In this section, we give such a criterion. We begin with the following lemma, which is true for triangulated categories:
 
\begin{lemma}\label{l:EnoughInjectivesImpliesSmal}
Let $\mathcal T$ be a triangulated subcategory of $\mathbf K(\Modr R)$ and $\mathcal D$ a thick subcategory of $\mathcal T$. Suppose that $\mathcal T$ has enough $\Mor_{\mathcal D}$-injective objects. Then the Verdier quotient $\mathcal T / \mathcal D$ has small Hom-sets.
\end{lemma}

\begin{proof}
Let $X$ and $Y$ be objects in $\mathcal T$. Take $\underline f:Y \rightarrow I$ a morphism in $\Mor_{\mathcal D}$ with $I$ being $\Mor_{\mathcal D}$-injective. Then $Y \cong I$ in $\mathcal T/\mathcal D$, so that we only have to see that $\Hom_{\mathcal T/\mathcal D}(X,I)$ is a set.

We prove that the canonical map $\phi:\Hom_{\mathcal T}(X,I) \rightarrow \Hom_{\mathcal T/\mathcal D}(X,I)$ is surjective. Take a morphism $(Z,s,f) \in \Hom_{\mathcal T/\mathcal D}(X,I)$. Since $s \in \Mor_{\mathcal D}$ and $I$ is $\Mor_{\mathcal D}$-injective, there exists $g:X \rightarrow I$ such that $gs=f$. Then, the following diagram is commutative:
\begin{displaymath}
\begin{tikzcd}
 & Z \arrow{dl}{s} \arrow{dr}{f} & \\
X & Z \arrow{l}{s} \arrow{u}{1_Z} \arrow{r}{f} \arrow{d}{s} & I\\
 & X \arrow{ul}{1_X} \arrow{ur}{g} & \\
\end{tikzcd}
\end{displaymath}
This implies that $\phi(g) = (X,1_X,g) = (Z,s,f)$ and, consequently, $\phi$ is surjective.
\end{proof}

In \cite[Theorem 4.4]{CortesGuilBerkeAshish} we proved the existence of enough injectives in certain exact categories. Our objective now is to adapt the proof of this result to the setting of this paper in order to give a necessary condition for the existence of enough $\Mor_{\mathcal K}$-injectives in $\mathbf K(\Proj)$. We will need a couple of lemmas. The first one is straightforward:

\begin{lemma}\label{l:TransfiniteSplitMono}
Let $\mathcal A$ be an additive category with direct limits. Let $(A_\alpha,u_{\alpha\beta}\mid \alpha < \beta < \kappa)$ be a transfinite sequence such that $u_{\alpha\alpha+1}$ is a split epimorphism (resp. split monomorphism) for each $\alpha < \kappa$. Then the transfinite composition of the transfinite sequence is a split epimorphism (resp. split monomorphism).
\end{lemma}

Given a transfinite sequence of complexes, $(X_\alpha,u_{\alpha\beta} \mid \alpha < \beta < \kappa)$, we define, for each non-zero $\alpha < \beta$ and $n \in \mathbb Z$, the morphism $c_{\alpha\beta}^n:\cone(u_{0\alpha})^n \rightarrow \cone(u_{0\beta})^n$ by $c_{\alpha\beta}^n(x,y) = (x,u_{\alpha\beta}^n(y))$ for each $(x,y) \in \cone(u_{0\alpha})^n$. It is very easy to see that these morphisms induce a cochain map $c_{\alpha\beta}:\cone(u_{0\alpha})\rightarrow \cone(u_{0\beta})$.

\begin{lemma}\label{l:ConeOfTransfiniteComposition}
Let $(X_\alpha,u_{\alpha\beta}\mid \alpha < \beta < \kappa)$ be a transfinite sequence of complexes with direct limit $(u_\alpha:X_\alpha \rightarrow X\mid \alpha < \kappa)$. Then the system $(\cone(u_{0\alpha}),c_{\alpha\beta}\mid 0< \alpha < \beta < \kappa)$, with the definition of $c_{\alpha\beta}$ made in the previous paragraph, is a transfinite sequence of complexes whose direct limit is $\cone(u_0)$.
\end{lemma}

\begin{proof}
Set $X_\kappa=X$ and $u_{\alpha\kappa}=u_\alpha$ for each $\alpha < \kappa$. Then $(X_\alpha,u_{\alpha\beta}\mid \alpha < \beta \leq \kappa)$ is a transfinite sequence. Now consider the system $(\cone(u_{0\alpha}),c_{\alpha\beta}\mid 0< \alpha < \beta \leq \kappa)$, with the $c_{\alpha\beta}$ defined as in the previous paragraph, which is easily seen to be direct. We only have to prove that it is continuous, that is, for every $\beta \leq \kappa$ limit, $(\cone(u_{0\beta}),c_{\alpha\beta}\mid 0 < \alpha < \beta)$ is the direct limit of $(\cone(u_{0\alpha},c_{\alpha \gamma}\mid 0 < \alpha < \gamma < \beta)$. Let $C$ be a complex and $(f_\alpha:\cone(u_{0\alpha}) \rightarrow C \mid \alpha < \beta)$ be a direct system of cochain maps. Denote by $i_\alpha:X_\alpha \rightarrow \cone(u_{0\alpha})$ the inclusion and by $g_\alpha=f_\alpha i_\alpha$ for each $\alpha < \beta$, we get a direct system of morphisms $(g_\alpha:X_\alpha \rightarrow C \mid 0 < \alpha < \beta)$. Since $(X_\alpha,u_{\alpha\gamma}\mid 0 < \alpha < \gamma < \beta)$ is a transfinite sequence, there exists a unique cochain map $g:X_\beta \rightarrow C$ satisfying $gu_{\alpha\beta} = g_\alpha$ for each $\alpha < \beta$. 

Now define $f^n:\cone(u_{0\beta})^n\rightarrow C^n$ by $f^n(x,y)=f_1^n(x,0)+g^n(y)$ for each $(x,y)\in \cone(u_{0\beta})^n$ and $n \in \mathbb Z$. Then, these $f^n$ define a cochain map $f:\cone(u_{0\beta})\rightarrow C$ since, for each $(x,y) \in \cone(u_{0\beta})$,
\begin{displaymath}
f^{n+1}d_{\cone(u_{0\beta})}^n(x,y) = f_1^{n+1}(-d_0^{n+1}(x),0)+g^{n+1}d_\beta^n(y)-g^{n+1}u_{0\beta}^{n+1}(y),
\end{displaymath}
and, using that $g^{n+1}u_{0\beta}^{n+1}(x)=g_1^{n+1}u_{01}(x)=f_1^{n+1}(0,u_{01}(x))$, we conclude that
\begin{eqnarray*}
f^{n+1}d_{\cone(u_{0\beta})}^n(x,y) = f_1^{n+1}(-d_0^{n+1}(x),0)-f_1^{n+1}(0,u_{01}(x))+g^{n+1}d_\beta^n(y)=\\ =f_1^{n+1}d_{\cone(u_{01})}^n(x,0)+g^{n+1}d_\beta^n(y).
\end{eqnarray*}
Now, using that both $g$ and $f_1$ are cochain maps, the last expression is equal to
\begin{displaymath}
d_C^nf_n^n(x,0)+d_C^ng^n(y) = d_C^nf^n(x,y)
\end{displaymath}
which means that $f$ is a cochain map. Moreover, it is easy to see that this $f$ is unique satisfying $f c_{\alpha\beta}=f_\alpha$ for each $\alpha < \beta$, which means that $(\cone(u_{0\beta}),c_{\alpha\beta}\mid 0 < \alpha < \beta)$ is the direct limit of $(\cone(u_{0\alpha},c_{\alpha \gamma}\mid 0 < \alpha < \gamma < \beta)$, as desired.
\end{proof}

\begin{lemma}\label{l:TransfiniteCompositionMor}
Let $(X_\alpha,u_{\alpha\beta}\mid \alpha < \beta < \kappa)$ be a transfinite sequence of complexes with transfinite composition $u$ such that $\underline u_{\alpha\beta}\in \Mor_{\mathcal K}$ and $u_{\alpha \beta}$ is a semi-split monomorphism for each $\alpha < \beta < \kappa$. Then $\underline u \in \Mor_{\mathcal K}$.
\end{lemma}

\begin{proof}
By Lemma \ref{l:ConeOfTransfiniteComposition}, $\cone(u)$ is the direct limit of the transfinite sequence\\ $(\cone(u_{0\alpha}),c_{\alpha\beta}\mid \alpha < \beta < \kappa)$, with $c_{\alpha\beta}$ as defined in the paragraph before Lemma \ref{l:ConeOfTransfiniteComposition}. Since $u_{\alpha\beta}$ is a semi-split monomorphism for each $\alpha < \beta$, then so is $c_{\alpha\beta}$. Now, using that $\cone(u_{0\alpha})$ and $\cone(u_{0\alpha+1})$ belong to $\mathcal K$ and that $c_{\alpha\alpha+1}$ is a semi-split monomorphisms, we conclude that $\Coker c_{\alpha\alpha+1} \in \mathcal K$. This says that $(\cone(u_{0\alpha}),c_{\alpha\beta}\mid \alpha < \beta < \kappa)$ actually is a filtration by totally acyclic complexes and, since the class of all totally acyclic complexes is closed under filtrations (see, for instance, \cite[Corollary 4.5]{CortesSaroch}), we get that $\underline u \in \Mor_{\mathcal K}$.
\end{proof}

\begin{remark}
Notice that the condition in the preceding lemma seems to be weaker than the property that every transfinite sequence $(X_\alpha,u_{\alpha\beta}\mid \alpha < \beta < \kappa)$ in $\mathbf C(\Modr R)$ with morphisms $\underline u_{\alpha\beta} \in \Mor_{\mathcal K}$ for each $\alpha < \beta < \kappa$ and transfinite composition $u$, satisfies that $\underline u \in \Mor_{\mathcal K}$. This is due to the fact that, although every morphism in the homotopy category can be represented by a semi-split monomorphism in $\mathbf C(\Modr R)$ by Lemma \ref{l:Semisplit_epimorphism}, it seems that this is not the case for transfinite sequences. That is, if $(X_\alpha,u_{\alpha\beta}\mid \alpha < \beta < \kappa)$ is a transfinite sequence in $\mathbf C(\Modr R)$, it seems that it might exist no transfinite sequence in $\mathbf C(\Modr R)$, $(Y_\alpha,v_{\alpha\beta}\mid \alpha < \beta < \kappa)$, with the same transfinite composition, with $v_{\alpha\beta}$ being a semi-split monomorphism and satisfying that $\underline u_{\alpha\beta}$ and $\underline v_{\alpha \beta}$ are isomorphic morphisms in $\mathbf K(\Modr R)$.
\end{remark}

Now we can prove the main result of our paper.

\begin{theorem}\label{t:ExistenceInjective}
Suppose that $\Mor_{\mathcal K}$ satisfies the generalized Baer Lemma in $\mathbf K(\Proj)$. Then there exists enough $\Mor_{\mathcal K}$-injective objects in $\mathbf K(\Proj)$. In particular:
\begin{enumerate}
\item The Verdier quotient $\mathbf K(\Proj)/\mathcal K$ has small Hom-sets.

\item The class of all Gorenstein-projective modules is precovering in $\Modr R$.

\item The class of all totally acyclic complexes is precovering in $\mathbf C(\Modr R)$.
\end{enumerate}
\end{theorem}

\begin{proof}
The procedure of the proof is the same as the one of \cite[Theorem 4.4]{CortesGuilBerkeAshish}, but with some modifications. 

First, notice that, reasoning as in the first paragraph of the proof of the aforementioned theorem, and using that $\Mor_{\mathcal K}$ is closed under direct sums and satisfies the generalized Baer Lemma, we can find a morphism $\underline m:\underline M \rightarrow \underline N$ in $\Mor_{\mathcal K}$ such that any $\underline m$-injective object of $\mathbf K(\Proj)$ is $\Mor_{\mathcal K}$-injective. We may assume by Lemma \ref{l:Semisplit_epimorphism} that $m$ is a semi-split monomorphism.

Fix $Y$ a complex in $\mathbf C(\Proj)$ and take an uncountable regular cardinal $\kappa$ such that $M^n$ is a $<\kappa$-generated module for each $n \in \mathbb Z$. We are going to construct a transfinite sequence in $\mathbf C(\Proj)$, $(E_\alpha,u_{\alpha\beta}\mid \alpha < \beta < \kappa\}$, such that $E_0=Y$, $u_{\alpha\beta}$ is a semi-split monomorphism, $\underline u_{\alpha\beta}\in \Mor_{\mathcal K}$ and 
\begin{enumerate}
\item[(\dag)] for every $\underline f:\underline M \rightarrow \underline E_\alpha$, there exists a $\underline g:\underline N \rightarrow \underline E_{\alpha+1}$ with $\underline u_{\alpha\alpha+1}\underline f=\underline g \underline m$.
\end{enumerate}

We make the construction recursively on $\beta$. If $\beta =0$, set $E_0=Y$. Suppose that we have made the construction for some $\beta < \kappa$ and set $I_\beta=\Hom_{\mathbf C(\Proj)}(M,E_\beta)$, $m_\beta:M^{(I_\beta)} \rightarrow N^{(I_\beta)}$ the direct sum of copies of $m$, and $v_\beta:M^{(I_\beta)} \rightarrow E_\beta$ the canonical morphism.  Making the pushout of $m_\beta$ and $v_\beta$ in $\mathbf C(\Modr R)$ we get the commutative diagram
\begin{equation}\label{e:Diagram}\tag{\dag\dag}
\begin{tikzcd}
M^{(I_\beta)} \arrow{r}{m_\beta} \arrow{d}{v_\beta} & N^{(I_\beta)} \arrow{d}{w}\\
E_\beta \arrow{r}{u} & W
\end{tikzcd}
\end{equation}
Since the pushout in $\mathbf C(\Modr R)$ is computed degree-wise, $u$ is a semi-split monomorphism and, in particular, $W \in \mathbf C(\Proj)$. Now, using that $\mathbf C(\Proj)$ is an exact category with the class of all semi-split short exact sequences, we can apply \cite[Proposition 2.11]{Buhler} to conclude that (\ref{e:Diagram}) is a pullback as well and, in particular, the induced sequence,
\begin{displaymath}
\begin{tikzcd}
0 \arrow{r} & M^{(I_\beta)} \arrow{r} & E_\beta \oplus N^{(I_\beta)} \arrow{r} & W \arrow{r} & 0
\end{tikzcd}
\end{displaymath}
is semi-split. Then, it induces an exact triangle in $\mathbf K(\Proj)$,
\begin{displaymath}
\begin{tikzcd}
M^{(I_\beta)} \arrow{r} & E_\beta \oplus N^{(I_\beta)} \arrow{r} & W \arrow{r} & M^{(I_\beta)}[-1].
\end{tikzcd}
\end{displaymath}
This means that (\ref{e:Diagram}) induces a homotopy pushout in $\mathbf K(\Proj)$ and, since $\underline m_\beta \in \Mor_{\mathcal K}$, $\underline u \in \Mor_{\mathcal K}$ as well by \cite[Lemma 1.5.8]{Neeman01}. Then set $u_{\beta\beta+1}=u$ and $E_{\beta\beta+1}=W$. From the commutativity of (\ref{e:Diagram}) and the definition of $v_\beta$ and $m_\beta$, it is easy to see that $\underline u_{\beta\beta+1}$ satisfies (\dag).

If $\beta$ is a limit ordinal, we compute the direct limit $(T,t_\alpha\mid \alpha < \beta)$ of $(E_{\alpha},u_{\alpha\gamma}\mid \alpha < \gamma < \beta)$ in $\mathbf C(\Modr R)$. Since $u_{\alpha\alpha+1}$ is a semi-split monomorphism, $(E_{\alpha}^n,u_{\alpha\gamma}^n \mid \alpha < \gamma < \beta)$ is a $\Proj$-filtration for every $n \in \mathbb Z$, so that $T \in \mathbf C(\Proj)$ by Kaplansy's theorem. Moreover, by Lemma \ref{l:TransfiniteSplitMono}, $t_\alpha$ is a semi-split monomorphisms for each $\alpha < \beta$ and $\underline t_\alpha$ belongs to $\Mor_{\mathcal K}$ by Lemma \ref{l:TransfiniteCompositionMor}. Then set $E_\beta=T$ and $u_{\alpha\beta}=t_\alpha$ for each $\alpha < \beta$. This finishes the construction.

Now, let $u:Y \rightarrow E$ be the transfinite composition in $\mathbf C(\Modr R)$ of the constructed transfinite sequence. By the same arguments as above, $E \in \mathbf C(\Proj)$ and $\underline u \in \Mor_{\mathcal K}$. It remains to see that $\underline E$ is $\Mor_{\mathcal K}$-injective. By the comments at the beginning of this proof, it is enough to prove that $\underline E$ is $\underline m$-injective. Denote by $u_\alpha:E_\alpha \rightarrow E$ the canonical morphism associated to the direct limit for each $\alpha < \kappa$ and take $\underline f:\underline M \rightarrow \underline E$ a morphism in $\mathbf K(\Proj)$. Since $M^n$ is $<\kappa$-generated for each $n \in \mathbb Z$ and $\kappa$ is uncountable regular, we can find an $\alpha < \kappa$ and a morphism $f_\alpha:M \rightarrow E_\alpha$ in $\mathbf C(\Proj)$ with $u_\alpha f_\alpha = f$. Applying (\dag), we can find $g:M \rightarrow E_{\alpha+1}$ such that $\underline u_{\alpha\alpha+1}\underline f_\alpha=\underline g\underline m$. It follows that $\underline f=\underline u_{\alpha+1}\underline g \underline m$, which implies that $\underline f$ has an extension and that $\underline E$ is $\underline m$-injective.

Finally, (1), (2) and (3) follow from Lemma \ref{l:EnoughInjectivesImpliesSmal}, Corollaries \ref{c:GorensteinPrecovers} and \ref{c:TotallyAcyclicPrecovers}.
\end{proof}

\bigskip

{\bf Acknowledgement.} The author thanks Peter J\o rgensen and Sergio Estrada for several stimulating conversations about the topic of the article.

\bibliographystyle{acm} \bibliography{/home/manolo/mizurdia@ual.es/Algebra/ReferenciasBibliograficas/references}

\end{document}